\documentclass[11pt,a4paper,leqno]{amsart}
\usepackage{amsmath,amssymb,amsthm,mathtools}
\usepackage[T1]{fontenc}
\usepackage[utf8]{inputenc}
\usepackage{lmodern}
\usepackage{microtype}
\usepackage{hyperref}
\hypersetup{
  colorlinks=true,
  linkcolor=blue,
  citecolor=blue,
  urlcolor=blue
}
\newtheorem{theorem}{Theorem}[section]
\newtheorem{proposition}[theorem]{Proposition}
\newtheorem{lemma}[theorem]{Lemma}
\newtheorem{corollary}[theorem]{Corollary}
\newtheorem{remark}[theorem]{Remark}
\theoremstyle{definition}

\theoremstyle{plain}
\numberwithin{equation}{section}

\newcommand{\R}{\mathbb{R}}

\title[Intrinsic Curvature Energy Estimates]
{Intrinsic Curvature Energy Estimates for Strongly Stable CMC Hypersurfaces}
\author{Anji Tang}
\address{School of Mathematical Sciences, Laboratory of Mathematics and Complex Systems\\
Beijing Normal University, Beijing 100875, P. R. CHINA.}
\email{ajtang@mail.bnu.edu.cn}
\date{}

\begin{document}

\begin{abstract}
We establish an intrinsic square-root curvature-energy estimate for
two-sided strongly stable CMC hypersurfaces in \(\R^{n+1}\) for
\(3\le n\le5\).  By means of a weighted H\"older closure, we control a
supercritical curvature norm in terms of the \(L^2\) norm of the
traceless second fundamental form, from which the pointwise bound
follows via Moser iteration.  In dimensions \(n=3\) and \(4\), the
estimate holds without any scale restriction on the mean curvature.
In regions of finite Morse index, this bound restricts points with low
principal curvatures to neighborhoods of at most as many centers as the
index, and yields an upper bound on the total diameter of separated
degeneration sets in terms of their curvature energy and the index.
\end{abstract}

\maketitle

\section{Introduction}
\label{sec:intro}

Based on Simons' identity~\cite{Simons68},
Schoen--Simon--Yau~\cite{SSY75} combined the curvature equation with
the stability inequality to derive \(L^p\) bounds for the second
fundamental form \(A\).  In Euclidean space, under the volume growth assumption
\(|B_R|\le\Lambda R^{n}\), their method gives the curvature estimate
\(\sup_{B_{\theta R}}|A|\le C(n,\theta,\Lambda)/R\) for
\(n\le5\), which remains valid on intrinsic geodesic balls~\cite[Theorem~3]{SSY75};
the corresponding regularity theory was later established by
Schoen--Simon~\cite{SS81}.  More recently,
Bellettini~\cite[Theorem~3]{Bel25} obtained a square-root
\(\varepsilon\)-regularity estimate for properly immersed two-sided
stable minimal hypersurfaces of dimension \(n\le6\).  Specifically,
if \(M\) is such a hypersurface in \(B_{2R}^{n+1}(0)\) with \(0\in M\)
satisfying
\[
\frac{1}{(2R)^{n-2}}\int_{M\cap B_{2R}^{n+1}(0)}|A_M|^2\,dv_M
\le\varepsilon_0(n),
\]
then
\[
\sup_{M\cap B_{R/2}^{n+1}(0)}|A_M|
\le\frac{c(n)}{R}
\left(\frac{1}{(2R)^{n-2}}
\int_{M\cap B_{2R}^{n+1}(0)}|A_M|^2\,dv_M\right)^{1/2}.
\]
In Bellettini's argument, the localization relies on ambient cutoff
functions, whose pullbacks to \(M\) are compactly supported by properness,
so both the \(L^2\) energy and the pointwise supremum are taken over
extrinsic balls.  A natural question is whether a similar energy estimate
holds on intrinsic geodesic balls for hypersurfaces of constant mean
curvature (CMC), with \(A\) replaced by its traceless part.

For a CMC hypersurface, it is natural to work with the traceless
second fundamental form
\[
\mathring A=A-Hg,\qquad H=\frac1n\operatorname{tr}_gA,
\]
whose norm satisfies \(|A|^2=|\mathring A|^2+nH^2\) and measures the
deviation of \(M\) from being totally umbilic~\cite{dCZ00}.  Integral
estimates built on the Nomizu--Smyth identity~\cite{NS69} and Okumura's
inequality~\cite{Oku74} were introduced by Alencar--do Carmo~\cite{AdC94}
and do Carmo--Zhou~\cite{dCZ00} in their study of Bernstein-type theorems.
General Caccioppoli-type inequalities for \(|\mathring A|\) were
subsequently derived by Ilias--Nelli--Soret~\cite{INS12}, while
Lu~\cite{Lu12} established \(L^p\) bounds for \(|\mathring A|\) together
with pointwise bounds under a local volume growth condition.  Curvature
estimates and sheeting theorems for weakly stable CMC hypersurfaces were
proved by Bellettini--Chodosh--Wickramasekera~\cite{BCW19}; see
also~\cite{CL24,CLMS26,Maz24,HLW26} for recent progress on the stable
Bernstein problem in Euclidean space and~\cite[Proposition~4.3]{CHL25}
for blow-up curvature estimates in the CMC setting.

Throughout the paper, \(M^n\) denotes an oriented two-sided hypersurface
immersed in \(\R^{n+1}\) with constant mean curvature \(H\), \(dv_M\) is
its induced volume element, and \(B_R(p_0)=B_R^M(p_0)\) denotes the
intrinsic geodesic ball of radius \(R\) centered at \(p_0\).  Recall that
\(M\) is \emph{strongly stable} if
\begin{equation}\label{eq:1.1}
\int_M|A|^2\phi^2\,dv_M\le\int_M|\nabla\phi|^2\,dv_M
\qquad\forall\,\phi\in C_c^\infty(M).
\end{equation}
When \eqref{eq:1.1} is assumed only for test functions \(\phi\) satisfying
\(\int_M\phi\,dv_M=0\), \(M\) is said to be weakly
stable~\cite{BdC84}.  For \(B_R(p_0)\Subset M\), we define the
scale-invariant traceless curvature energy by
\[
E_R=R^{2-n}\int_{B_R(p_0)}|\mathring A|^2\,dv_M .
\]

Our first main result is the following local \(L^2\)-to-\(L^\infty\)
estimate on intrinsic balls.

\begin{theorem}\label{thm:main}
Let \(M^n\subset\R^{n+1}\), \(3\le n\le5\), be an oriented two-sided
strongly stable CMC hypersurface, and let \(B_R(p_0)\Subset M\).
Assume \(E_R\le1\); when \(n=5\), assume in addition that \(|H|R\le1\).
Then for every \(0<\theta<1\),
\begin{equation}\label{eq:main}
R\sup_{B_{\theta R}(p_0)}|\mathring A|\le C(n,\theta)\,E_R^{1/2}.
\end{equation}
\end{theorem}

Note that for \(n=3,4\), no upper bound on \(|H|R\) is needed.  To prove
Theorem~\ref{thm:main}, we start from the weighted Caccioppoli inequality
of do Carmo--Zhou~\cite{dCZ00} and Ilias--Nelli--Soret~\cite{INS12}.
Taking a cutoff of the form \(\psi=f^{1+q}\), we write
\[
 u^{2+2q}f^{2q}|\nabla f|^2
 =\bigl(u^{4+2q}f^{2+2q}\bigr)^{q/(1+q)}
  \bigl(u^2|\nabla f|^{2+2q}\bigr)^{1/(1+q)},
\]
so that H\"older's inequality allows us to absorb the higher power of
\(u\) and bound \(\int u^{4+2q}f^{2+2q}\) directly by the \(L^2\)
energy of \(u=|\mathring A|\), without coupling the cutoff scale to the
mean curvature \(H\).  When \(n=3\) or \(4\), the algebraic structure of
the Caccioppoli coefficients allows us to absorb the mixed
mean-curvature term completely into the positive \(H^2\) term at
\(q=1/4\).  Once this supercritical \(L^{2+q}\) bound for
\(w=|\mathring A|^2\) is in hand, a standard Moser iteration followed by
an interpolation argument yields~\eqref{eq:main}.

As an application of Theorem~\ref{thm:main}, we study the localization of
small principal curvatures on regions of finite Morse index.  For an open
set \(U\subset M\), the strong Morse index \(\operatorname{ind}(U)\) is
defined as the maximal dimension of a subspace of \(C_c^\infty(U)\) on
which the quadratic form \(Q(\phi)=\int_U(|\nabla\phi|^2-|A|^2\phi^2)\)
is negative definite.  Choosing the unit normal so that \(H>0\) and
ordering the principal curvatures as \(\lambda_1\le\cdots\le\lambda_n\),
we consider the sublevel sets
\[
 \mathcal D_{m,\eta}(K)=\{x\in K:\lambda_m(x)\le\eta H\},
 \qquad 1\le m<n,\quad 0\le\eta<1.
\]

\begin{theorem}\label{thm:1.2}
Let \(M^n\to\R^{n+1}\), \(3\le n\le5\), be a smooth two-sided
CMC immersion.  Suppose \(B_R(p)\Subset M\), \(0<HR\le1\), and
\(\operatorname{ind}(B_R(p))\le I\) for some \(I\in\mathbb N_0\).
Set \(K=\overline{B_{R/2}(p)}\) and
\[
 E_R=R^{2-n}\int_{B_R(p)}|\mathring A|^2\,dv_M,\qquad
 \kappa_n=\min\{C(n,1/2)^{-2},[n(n-1)]^{-1}\},
\]
where \(C(n,1/2)\) is the constant from Theorem~\ref{thm:main}, and write
\(u=|\mathring A|\).
Then there exists a finite set \(\mathcal S\subset K\) with
\(\#\mathcal S\le I\), depending only on the geometry of \(B_R(p)\) and
independent of \(m\), \(\eta\), and the observation radius, such that
for every \(x\in\mathcal D_{m,\eta}(K)\setminus\mathcal S\) and every
\(0<r\le\frac18\min\{R,d_M(x,\mathcal S)\}\),
\begin{equation}\label{eq:1.3}
\int_{B_r(x)}u^2\,dv_M
\ge\kappa_n\frac{nm}{n-m}(1-\eta)^2H^2r^n.
\end{equation}
Moreover, if we set
\begin{equation}\label{eq:1.4}
\ell_{m,\eta}
=R\left(
\frac{n-m}{\kappa_n nm(1-\eta)^2}
\frac{E_R}{(HR)^2}
\right)^{1/n}
\end{equation}
and assume \(\ell_{m,\eta}<R/8\), then
\begin{equation}\label{eq:1.5}
\mathcal D_{m,\eta}(K)
\subset\bigcup_{a\in\mathcal S}
\{x\in K:d_M(x,a)\le8\ell_{m,\eta}\}.
\end{equation}
In particular, taking \(m=1\) and \(\eta=0\), the second fundamental
form \(A\) is positive definite on \(K\) outside the \(8\ell_{1,0}\)-neighborhood
of \(\mathcal S\); if \(I=0\) and \(\ell_{m,\eta}<R/8\), then
\(\lambda_m>\eta H\) everywhere on \(K\).

\end{theorem}

The scale assumption \(HR\le1\) in Theorem~\ref{thm:1.2} remains
necessary for the proof given here in every dimension \(3\le n\le5\).
Although Theorem~\ref{thm:main} does not require it when \(n=3,4\),
the high-local-energy branch in Section~\ref{sec:4} uses
\(r^{n-2}\ge H^2r^n\), which follows from \(Hr\le1\).
Thus the two main theorems have different scale hypotheses; no global
smallness assumption on \(E_R\) is imposed in Theorem~\ref{thm:1.2}.

The idea behind Theorem~\ref{thm:1.2} is that a bound on the Morse index
guarantees strong stability on geodesic balls sufficiently far from a
finite set of centers \(\mathcal S\).  At any point in
\(\mathcal D_{m,\eta}\), an elementary algebraic inequality gives the
pointwise lower bound \(|\mathring A|^2\ge nm(1-\eta)^2H^2/(n-m)\), which
by Theorem~\ref{thm:main} forces a definite amount of \(L^2\) curvature
energy on any surrounding stable ball.  Packing such balls along
connected components of \(\mathcal D_{m,\eta}\) leads to
Corollary~\ref{cor:4.3}, where we bound both the diameter of a single
connected degeneration set and the total diameter of a family of
separated degeneration sets in terms of \(E_R\) and the index.

We remark that an interior curvature bound for strongly stable CMC
hypersurfaces was also obtained recently by Miranda~\cite[Theorem~2.3]{Mir26}
via a point-selection and blow-up argument; when combined with the
differential inequality for \(|\mathring A|^2\) and a standard mean-value
inequality, that bound can also be used to recover~\eqref{eq:main}.  The
direct integral approach presented here avoids blow-up arguments and yields
explicit weighted \(L^p\) estimates along the way.

The paper is organized as follows.  In Section~\ref{sec:2}, we derive the
weighted integral estimates for \(|\mathring A|\), including the
\(H\)-independent bound in dimensions \(3\) and \(4\).  In
Section~\ref{sec:3}, we carry out the Moser iteration to prove
Theorem~\ref{thm:main} and check the sharpness of the exponent \(1/2\).
Finally, Section~\ref{sec:4} contains the proofs of Theorem~\ref{thm:1.2}
and Corollary~\ref{cor:4.3}.

\section{Integral Estimates for the Traceless Curvature}
\label{sec:2}

Throughout this section, \(M^n\to\R^{n+1}\) (\(n\ge2\)) is a
two-sided strongly stable CMC immersion.  We write
\[
 \mathring A=A-Hg,\qquad u=|\mathring A|,
 \qquad |A|^2=u^2+nH^2.
\]
By standard approximation, compactly supported Lipschitz functions are
admissible test functions in the stability inequality~\eqref{eq:1.1}.
All estimates below are local: on a strongly stable open set, the same
proofs apply to test functions compactly supported in that set. When a
radial cutoff vanishes at the boundary of a reference ball, it can first
be cut off at a strictly smaller outer radius and then recovered by a
limit. This observation will justify applying Theorem~\ref{thm:main}
inside a stable ball in Section~\ref{sec:4}, without assuming stability
of the entire immersion.
Constants denoted by \(C(n,q)\) depend only on the indicated parameters.

Combining the Simons identity for \(\mathring A\) with Okumura's
inequality and the refined Kato inequality
\(|\nabla\mathring A|^2\ge(1+2/n)|\nabla u|^2\) gives the well-known
differential inequality for \(u\) (see~\cite[Eq.~(1.5)]{dCZ00}
and~\cite[Lemma~2.1]{Lu12}).

\begin{lemma}\label{lem:2.1}
On the open set \(\{u>0\}\), we have
\begin{equation}\label{eq:2.2}
 u\Delta u+u^4+\alpha|H|u^3-nH^2u^2
 \ge\frac2n|\nabla u|^2,
 \qquad \alpha=\frac{n(n-2)}{\sqrt{n(n-1)}},
\end{equation}
where \(\Delta=\operatorname{div}\nabla\) is the Laplace--Beltrami operator on \(M\).
\end{lemma}

To handle the zero set of \(u\) in integration by parts, one replaces
\(u\) by the smooth regularization
\(u_\varepsilon=(u^2+\varepsilon^2)^{1/2}\) with \(0<\varepsilon\le1\).
Then
\[
 \nabla u_\varepsilon=\frac{u}{u_\varepsilon}\nabla u,
 \qquad |\nabla u_\varepsilon|\le|\nabla\mathring A|
\]
almost everywhere, with \(\nabla u_\varepsilon=0\) on \(\{u=0\}\).
From the tensor Simons identity and the refined Kato inequality, we have
\[
\begin{aligned}
 u_\varepsilon\Delta u_\varepsilon
 &=\tfrac12\Delta(u^2)-|\nabla u_\varepsilon|^2\\
 &\ge\frac2n|\nabla u_\varepsilon|^2-u^4-\alpha|H|u^3+nH^2u^2\\
 &\ge\frac2n|\nabla u_\varepsilon|^2-u_\varepsilon^4
 -\alpha|H|u_\varepsilon^3+nH^2u_\varepsilon^2-nH^2\varepsilon^2,
\end{aligned}
\]
where the first inequality holds across \(\{u=0\}\) since both the
curvature terms and \(\nabla u_\varepsilon\) vanish there.  In terms of
\(u_\varepsilon\), the stability inequality~\eqref{eq:1.1} reads
\[
 \int(u_\varepsilon^2+nH^2)\varphi^2
 \le\int|\nabla\varphi|^2+\varepsilon^2\int\varphi^2.
\]
Testing this with \(\varphi=u_\varepsilon^{1+q}\psi\) produces remainder
terms controlled by
\[
 \varepsilon^2\int_{\operatorname{supp}\psi}u_\varepsilon^{2+2q}\psi^2
 \quad\text{and}\quad
 H^2\varepsilon^2\int_{\operatorname{supp}\psi}u_\varepsilon^{2q}\psi^2.
\]
Since \(q\ge0\) and \(u\) is bounded on \(\operatorname{supp}\psi\),
both remainders vanish as \(\varepsilon\downarrow0\), while the gradient
integrals are dominated by multiples of \(|\nabla\mathring A|^2\) on
\(\operatorname{supp}\psi\).  By dominated convergence, we may therefore
pass to the limit \(\varepsilon\downarrow0\) and perform the formal
calculations below directly with \(u\).

Fix \(0<q<\sqrt{2/n}\) and set
\[
 L=\frac2n+q,\qquad t=\frac{2/n-q^2}{2(q+2)}>0.
\]
The weighted Caccioppoli inequality~\cite[Eq.~(3.1)]{dCZ00},
\cite[Theorem~4.1, Eq.~(41)]{INS12} takes the form
\begin{equation}\label{eq:original-caccioppoli}
 \int u^{2+2q}\psi^2(a_0u^2-b_0|H|u+c_0H^2)
 \le d_0\int u^{2+2q}|\nabla\psi|^2,
\end{equation}
where the coefficients (after multiplying those in~\cite{dCZ00} by \(L\))
are given by
\begin{equation}\label{eq:coefficients}
\begin{aligned}
 a_0&=L-q(1+q+t),& b_0&=\alpha(1+q)(1+q+t),\\
 c_0&=n[L+(2+q)(1+q+t)],&
 d_0&=\frac{1+q+t}{t}(L+t).
\end{aligned}
\end{equation}
Notice that for \(0<q<\sqrt{2/n}\),
\[
 a_0=\left(\frac2n-q^2\right)\frac{q+4}{2(q+2)}>0,
 \qquad c_0,d_0>0,\qquad b_0\ge0.
\]
For completeness, we briefly recall how~\eqref{eq:original-caccioppoli}
is obtained.  Set \(h=|H|\), suppress the volume measure \(dv_M\), and
introduce the shorthand notations
\[
 \begin{gathered}
 I=\int u^{2q}|\nabla u|^2\psi^2,\qquad
 J=\int u^{2q+1}\psi\langle\nabla u,\nabla\psi\rangle,\\
 K=\int u^{2q+2}|\nabla\psi|^2,\qquad
 U=\int u^{2q+4}\psi^2,\qquad
 V=\int u^{2q+3}\psi^2,\qquad W=\int u^{2q+2}\psi^2.
 \end{gathered}
\]
For any compactly supported Lipschitz function \(\psi\), all six integrals
are finite.  Multiplying~\eqref{eq:2.2} by \(u^{2q}\psi^2\) and
integrating by parts gives
\[
 \begin{aligned}
 \int u^{2q+1}\psi^2\Delta u
 &=-\int\left\langle
 (2q+1)u^{2q}\psi^2\nabla u+2u^{2q+1}\psi\nabla\psi,
 \nabla u\right\rangle\\
 &=-(2q+1)I-2J,
 \end{aligned}
\]
and hence
\[
 (1+2q+2/n)I+2J\le U+\alpha hV-nh^2W.
\]
On the other hand, testing the stability inequality~\eqref{eq:1.1} with
\(\phi=u^{1+q}\psi\) yields
\[
 U+nh^2W\le(1+q)^2I+2(1+q)J+K.
\]
Eliminating \(2J\) between these two inequalities, we find that the
coefficient of \(I\) on the right-hand side is
\[
 (1+q)^2-(1+q)(1+2q+2/n)=-(1+q)(q+2/n)=-(1+q)L,
\]
which rearranges to
\[
 L(1+q)I\le qU+(1+q)\alpha hV-(2+q)nh^2W+K.
\]
Next, we return to the stability inequality and estimate the cross term
\(2J\) via Cauchy--Schwarz:
\[
 2|J|\le
 t\int u^{2q}|\nabla u|^2\psi^2
 +t^{-1}\int u^{2q+2}|\nabla\psi|^2=tI+t^{-1}K,
\]
so that
\[
 U+nh^2W
 \le(1+q)(1+q+t)I+\left(1+\frac{1+q}{t}\right)K.
\]
Multiplying this inequality by \(L>0\), substituting the upper bound for
\(L(1+q)I\) found above, and collecting like terms in \(U,V,W\), we arrive at
\[
 \begin{aligned}
 &[L-q(1+q+t)]U-\alpha(1+q)(1+q+t)hV\\
 &\quad+n[L+(2+q)(1+q+t)]h^2W\\
 &\qquad\le
 \left[(1+q+t)+L\left(1+\frac{1+q}{t}\right)\right]K.
 \end{aligned}
\]
Since the coefficient of \(K\) simplifies to
\[
 (1+q+t)+L\frac{t+1+q}{t}
 =\frac{1+q+t}{t}(L+t)=d_0,
\]
we obtain~\eqref{eq:original-caccioppoli} with the coefficients
in~\eqref{eq:coefficients}.  

\begin{proposition}\label{prop:2.2}
For every compactly supported Lipschitz function \(\psi\),
\begin{equation}\label{eq:2.5}
 \int u^{4+2q}\psi^2
 \le L_1\int u^{2+2q}|\nabla\psi|^2
        +L_2H^2\int u^{2+2q}\psi^2,
\end{equation}
where
\begin{equation}\label{eq:2.6}
 L_1=\frac{2d_0}{a_0},\qquad
 L_2=\max\left\{0,\frac{b_0^2}{a_0^2}-\frac{2c_0}{a_0}\right\}.
\end{equation}
\end{proposition}
\begin{proof}
By Cauchy's inequality,
\[
 b_0|H|u\le\frac{a_0}{2}u^2+\frac{b_0^2}{2a_0}H^2.
\]
Inserting this into~\eqref{eq:original-caccioppoli} gives
\[
 \frac{a_0}{2}\int u^{4+2q}\psi^2
 \le d_0\int u^{2+2q}|\nabla\psi|^2
 +\left(\frac{b_0^2}{2a_0}-c_0\right)H^2
                            \int u^{2+2q}\psi^2.
\]
Dividing by \(a_0/2>0\) and replacing the coefficient of the last integral
by its positive part yields~\eqref{eq:2.5}.
\end{proof}

\begin{proposition}\label{prop:2.3}
For every nonnegative compactly supported Lipschitz function \(f\),
\begin{equation}\label{eq:2.7}
 \int u^{4+2q}f^{2+2q}
 \le\mathcal C_1\int u^2|\nabla f|^{2+2q}
        +\mathcal C_2|H|^{2+2q}\int u^2f^{2+2q},
\end{equation}
where
\begin{equation}\label{eq:2.8}
 \mathcal C_1=2^q[(1+q)^2L_1]^{1+q},\qquad
 \mathcal C_2=2^qL_2^{1+q}.
\end{equation}
\end{proposition}
\begin{proof}
Set \(\psi=f^{1+q}\) in Proposition~\ref{prop:2.2}, so that
\(|\nabla\psi|^2=(1+q)^2f^{2q}|\nabla f|^2\).
Defining
\[
 X=\int u^{4+2q}f^{2+2q},\quad
 Y=\int u^2|\nabla f|^{2+2q},\quad Z=\int u^2f^{2+2q},
\]
we can rewrite~\eqref{eq:2.5} as
\[
 X\le(1+q)^2L_1\int u^{2+2q}f^{2q}|\nabla f|^2
            +L_2H^2\int u^{2+2q}f^{2+2q}.
\]
Because \((4+2q)q+2=(2+2q)(1+q)\), the integrands on the right-hand side
factor as
\[
 u^{2+2q}f^{2q}|\nabla f|^2
 =(u^{4+2q}f^{2+2q})^{q/(1+q)}(u^2|\nabla f|^{2+2q})^{1/(1+q)}
\]
and
\[
 u^{2+2q}f^{2+2q}
 =(u^{4+2q}f^{2+2q})^{q/(1+q)}(u^2f^{2+2q})^{1/(1+q)}.
\]
Applying H\"older's inequality with conjugate exponents \((1+q)/q\) and
\(1+q\), we obtain
\[
\begin{aligned}
 \int u^{2+2q}f^{2q}|\nabla f|^2
 &\le X^{q/(1+q)}Y^{1/(1+q)},\\
 \int u^{2+2q}f^{2+2q}
 &\le X^{q/(1+q)}Z^{1/(1+q)}.
\end{aligned}
\]
Assuming without loss of generality that \(X>0\), we divide through by
\(X^{q/(1+q)}\) to get
\[
 X^{1/(1+q)}\le(1+q)^2L_1Y^{1/(1+q)}+L_2H^2Z^{1/(1+q)}.
\]
Raising both sides to the power \(1+q\) and using the elementary
inequality \((a+b)^{1+q}\le2^q(a^{1+q}+b^{1+q})\) gives~\eqref{eq:2.7}.
\end{proof}

When \(H=0\), \eqref{eq:2.7} reduces to~\cite[Theorem~5.5]{INS12} with
\(\mu=2+q\) and \(\eta_0=(1+q)/(2+q)\).  For \(H\ne0\), choosing the
exponent \(1+q\) on \(f\) ensures that \(u^2\) appears in both \(Y\) and
\(Z\), rather than producing a pure volume term \(\int f^{2+2q}\).
Specializing \(f\) to a standard radial cutoff gives the following local
estimate on geodesic balls.

\begin{theorem}\label{thm:2.4}
Let \(B_R(p_0)\Subset M\), \(0<\theta<1\), and
\(0<q<\sqrt{2/n}\).  Then
\begin{equation}\label{eq:2.9}
 \int_{B_{\theta R}(p_0)}u^{4+2q}
 \le\left[\frac{\mathcal C_1}{(1-\theta)^{2+2q}R^{2+2q}}
                  +\mathcal C_2|H|^{2+2q}\right]
              \int_{B_R(p_0)}u^2,
\end{equation}
with \(\mathcal C_1,\mathcal C_2\) given by~\eqref{eq:2.8}.
\end{theorem}
\begin{proof}
Let \(f\) be the radial cutoff function
\[
 f(x)=\begin{cases}
 1,&d_M(x,p_0)\le\theta R,\\
 \dfrac{R-d_M(x,p_0)}{(1-\theta)R},&\theta R<d_M(x,p_0)<R,\\
 0,&d_M(x,p_0)\ge R.
 \end{cases}
\]
Since \(d_M(\cdot,p_0)\) is \(1\)-Lipschitz, \(f\) is a compactly
supported Lipschitz function on \(M\) with \(f\equiv1\) on
\(B_{\theta R}(p_0)\), \(0\le f\le1\), and
\(|\nabla f|\le[(1-\theta)R]^{-1}\) almost everywhere on \(B_R(p_0)\).
Substituting \(f\) into Proposition~\ref{prop:2.3} immediately
gives~\eqref{eq:2.9}. If stability is known only on \(B_R(p_0)\),
use instead a cutoff equal to one on \(B_{\theta R}(p_0)\), decreasing
linearly to zero at \(R'<R\), where \(\theta R<R'\), and then let
\(R'\uparrow R\). Its support is compact in \(B_R(p_0)\), and the
gradient bound tends to \([(1-\theta)R]^{-1}\). Boundedness of \(u\)
on compact subsets and monotone convergence for the ball integrals give
the same estimate, with the same constants.
\end{proof}

In dimensions \(n=3\) and \(4\), the positive \(c_0H^2\) term in~\eqref{eq:original-caccioppoli}
is large enough to dominate the mixed term \(b_0|H|u\), which allows us
to drop the mean-curvature term in~\eqref{eq:2.9} altogether.

\begin{corollary}\label{cor:H-free}
If \(n=3,4\), then for any \(B_R(p_0)\Subset M\) and \(0<\theta<1\),
\begin{equation}\label{eq:H-free-integral}
 \int_{B_{\theta R}(p_0)}u^{9/2}
 \le\frac{D_n}{(1-\theta)^{5/2}R^{5/2}}
            \int_{B_R(p_0)}u^2,
\end{equation}
where \(D_n\) depends only on \(n\) and is independent of \(H\).
\end{corollary}
\begin{proof}
Take \(q=1/4\) in~\eqref{eq:coefficients}.  A direct check for \(n=3,4\)
shows that
\[
 a_0>\frac25,\qquad b_0^2<16,\qquad c_0>12,
 \qquad a_0-\frac{b_0^2}{4c_0}>\frac1{15}.
\]
Completing the square with respect to \(|H|\) in~\eqref{eq:original-caccioppoli}
gives
\[
 a_0u^2-b_0|H|u+c_0H^2
 =c_0\left(|H|-\frac{b_0}{2c_0}u\right)^2
   +\left(a_0-\frac{b_0^2}{4c_0}\right)u^2
 \ge\frac1{15}u^2,
\]
Thus we obtain a stronger version of Proposition~\ref{prop:2.2},
with coefficients \(\widetilde L_1=15d_0\) and \(\widetilde L_2=0\).
These are new admissible coefficients for this estimate, not the values
of \(L_1,L_2\) defined in~\eqref{eq:2.6}. In particular,
\(\int u^{9/2}\psi^2\le15d_0\int u^{5/2}|\nabla\psi|^2\).
Setting \(\psi=f^{5/4}\), \(X=\int u^{9/2}f^{5/2}\), and
\(Y=\int u^2|\nabla f|^{5/2}\), the H\"older argument of
Proposition~\ref{prop:2.3} yields
\[
 X\le\frac{375d_0}{16}X^{1/5}Y^{4/5},
 \qquad\text{so}\qquad
 X\le\left(\frac{375d_0}{16}\right)^{5/4}Y.
\]
If \(X=0\), the conclusion is immediate; otherwise division by
\(X^{1/5}\) justifies the last step. Choosing \(f\) as in the proof of
Theorem~\ref{thm:2.4} completes the proof, with
\[
 D_n=\left(\frac{375d_0}{16}\right)^{5/4},
 \qquad q=\frac14.
\]
\end{proof}

\section{Moser Iteration and Pointwise Curvature Control}
\label{sec:3}

We now set \(w=u^2\) and \(\chi=n/(n-2)\), and run a Moser iteration for
\(w\).  To start the iteration, we need a supercritical integrability
exponent \(s=2+q>n/2\) within the range \(0<q<\sqrt{2/n}\) from
Section~\ref{sec:2}, that is,
\begin{equation}\label{eq:3.1}
 \max\{0,(n-4)/2\}<q<\sqrt{2/n},\qquad s=2+q>n/2.
\end{equation}
For \(n>2\), such a \(q\) exists precisely when \(3\le n\le5\);
concretely, we shall take \(q=1/4\) for \(n=3,4\) and \(q=3/5\) for \(n=5\).

We first record the \(L^2\) Sobolev inequality on strongly stable CMC
hypersurfaces, derived from the Michael--Simon--Allard
inequality~\cite{All72,MS73}.

\begin{theorem}\label{thm:3.1}
Let \(M^n\to\R^{n+1}\), \(n>2\), be an immersed CMC hypersurface.
There exists a constant \(S_n\) depending only on \(n\) such that
\begin{equation}\label{eq:3.2}
 \|\phi\|_{L^{2\chi}}^2
 \le S_n\int_M(|\nabla\phi|^2+H^2\phi^2)
 \qquad\forall\,\phi\in C_c^1(M).
\end{equation}
If \(M\) is strongly stable, then for every compactly supported Lipschitz
function \(\phi\),
\begin{equation}\label{eq:pure-sobolev}
 \|\phi\|_{L^{2\chi}}^2\le\widehat S_n\int_M|\nabla\phi|^2,
 \qquad \widehat S_n=S_n(1+1/n).
\end{equation}
\end{theorem}
\begin{proof}
By the Michael--Simon--Allard Sobolev inequality~\cite{All72,MS73}, for
any compactly supported Lipschitz function \(v\) on \(M\),
\[
 \left(\int_M |v|^{n/(n-1)}\right)^{(n-1)/n}
 \le C_n^{\mathrm{MS}}\int_M
       \bigl(|\nabla v|+n|H||v|\bigr),
\]
where the factor \(n\) reflects our convention \(H=\operatorname{tr}A/n\).
Applying this to \(v=|\phi|^r\) with
\[
 r=\frac{2(n-1)}{n-2},\qquad
 r\frac n{n-1}=2\chi,\qquad r-1=\chi,
\]
and using the Cauchy--Schwarz inequality, we have
\[
 \begin{aligned}
 \int|\nabla(|\phi|^r)|
 &\le r\int|\phi|^\chi|\nabla\phi|
 \le r\|\phi\|_{2\chi}^\chi\|\nabla\phi\|_2,\\
 \int|\phi|^r
 &=\int|\phi|^\chi|\phi|
 \le\|\phi\|_{2\chi}^\chi\|\phi\|_2.
 \end{aligned}
\]
Assuming \(\phi\not\equiv0\), we divide by \(\|\phi\|_{2\chi}^\chi\) and
use \(r-\chi=1\) to obtain
\[
 \|\phi\|_{2\chi}
 \le C_n^{\mathrm{MS}}
       \bigl(r\|\nabla\phi\|_2+n|H|\|\phi\|_2\bigr),
\]
which upon squaring gives~\eqref{eq:3.2} with
\(S_n=2(C_n^{\mathrm{MS}})^2\max\{r^2,n^2\}\).

When \(M\) is strongly stable, the pointwise inequality
\(|A|^2=u^2+nH^2\ge nH^2\) together with~\eqref{eq:1.1} implies
\[
 H^2\int\phi^2\le\frac1n\int|A|^2\phi^2
 \le\frac1n\int|\nabla\phi|^2.
\]
Substituting this into~\eqref{eq:3.2} gives~\eqref{eq:pure-sobolev} for
\(\phi\in C_c^1(M)\), and the extension to compactly supported Lipschitz
functions follows by standard density.
\end{proof}

\begin{lemma}\label{lem:3.2}
On any CMC hypersurface, \(w=|\mathring A|^2\) satisfies
\begin{equation}\label{eq:3.3}
 \Delta w\ge-c_nw^2,\qquad
 c_n=2+\frac{\alpha^2}{2n},\quad
 \alpha=\frac{n(n-2)}{\sqrt{n(n-1)}}.
\end{equation}
\end{lemma}
\begin{proof}
By the Simons identity for \(\mathring A\) and Okumura's
inequality~\cite{NS69,Oku74},
\[
 \frac12\Delta w
 =|\nabla\mathring A|^2+nH\operatorname{tr}(\mathring A^3)
       +nH^2w-w^2
 \ge-w^2-\alpha|H|w^{3/2}+nH^2w.
\]
Completing the square in \(|H|\), we have
\[
 2nH^2w-2\alpha|H|w^{3/2}
 =2nw\left(|H|-\frac{\alpha}{2n}\sqrt w\right)^2
                  -\frac{\alpha^2}{2n}w^2
 \ge-\frac{\alpha^2}{2n}w^2,
\]
which yields~\eqref{eq:3.3}.
\end{proof}

\begin{lemma}\label{lem:3.3}
Let \(w\ge0\) be a smooth function satisfying~\eqref{eq:3.3} on
\(B\Subset M\).  For every \(p\ge2\) and every compactly supported
Lipschitz function \(\eta\) on \(B\),
\begin{equation}\label{eq:3.4}
 \int|\nabla(\eta w^{p/2})|^2
 \le C_np^2\int w^p|\nabla\eta|^2
                      +C_np\int\eta^2w^{p+1},
\end{equation}
where \(C_n\) depends only on \(n\).
\end{lemma}
\begin{proof}
Set \(v=w+\varepsilon>0\) for \(\varepsilon>0\).  Since
\(-\Delta v=-\Delta w\le c_nw^2\le c_nv^2\), multiplying by
\(\eta^2v^{p-1}\) and integrating by parts gives
\[
 (p-1)\int\eta^2v^{p-2}|\nabla v|^2
 +2\int\eta v^{p-1}\langle\nabla\eta,\nabla v\rangle
 \le c_n\int\eta^2v^{p+1}.
\]
By Cauchy's inequality,
\[
 2\int |\eta|v^{p-1}|\nabla\eta||\nabla v|
 \le\frac{p-1}{2}\int\eta^2v^{p-2}|\nabla v|^2
     +\frac2{p-1}\int v^p|\nabla\eta|^2,
\]
so that
\[
 \int\eta^2v^{p-2}|\nabla v|^2
 \le\frac4{(p-1)^2}\int v^p|\nabla\eta|^2
      +\frac{2c_n}{p-1}\int\eta^2v^{p+1}.
\]
Since \(\nabla(\eta v^{p/2})=v^{p/2}\nabla\eta+\frac p2\eta v^{(p-2)/2}\nabla v\),
it follows that
\[
 \begin{aligned}
 \int|\nabla(\eta v^{p/2})|^2
 &\le2\int v^p|\nabla\eta|^2
      +\frac{p^2}{2}\int\eta^2v^{p-2}|\nabla v|^2\\
 &\le\left(2+\frac{2p^2}{(p-1)^2}\right)
      \int v^p|\nabla\eta|^2
      +\frac{c_np^2}{p-1}\int\eta^2v^{p+1}.
 \end{aligned}
\]
Because \(p\ge2\), we have \(p/(p-1)\le2\) and
\(2+2p^2/(p-1)^2\le10\le3p^2\), so~\eqref{eq:3.4} follows with
\(C_n=\max\{3,2c_n\}\) after letting \(\varepsilon\downarrow0\).
\end{proof}

\begin{proposition}\label{prop:3.4}
Suppose \(n>2\), \(s>n/2\), and~\eqref{eq:pure-sobolev} holds.
Let \(w\ge0\) be a smooth function satisfying~\eqref{eq:3.3} on
\(B_\rho(p_0)\Subset M\).  Then for any \(0<\rho'<\rho\) and \(p\ge2\),
\begin{equation}\label{eq:3.6}
 \|w\|_{L^{p\chi}(B_{\rho'})}
 \le\left[Cp^m\left((\rho-\rho')^{-2}
              +\|w\|_{L^s(B_\rho)}^{\frac{2s}{2s-n}}\right)\right]^{1/p}
          \|w\|_{L^p(B_\rho)},
\end{equation}
where \(m=\max\{2,2s/(2s-n)\}\) and \(C=C(n,s)\).
\end{proposition}
\begin{proof}
Let \(\eta\) be a radial Lipschitz cutoff supported in \(B_\rho\) with
\(0\le\eta\le1\), \(\eta\equiv1\) on \(B_{\rho'}\), and
\(|\nabla\eta|\le2/(\rho-\rho')\).  Setting \(\psi=\eta w^{p/2}\), we
combine~\eqref{eq:pure-sobolev} and~\eqref{eq:3.4} to obtain
\begin{equation}\label{eq:3.7}
 \|\psi\|_{L^{2\chi}}^2
 \le Cp^2(\rho-\rho')^{-2}\int_{B_\rho}w^p
      +Cp\int_{B_\rho}w\psi^2.
\end{equation}
Let \(\beta=n/(2s)\in(0,1)\).  Since
\(\frac{s-1}{2s}=\frac\beta{2\chi}+\frac{1-\beta}{2}\), H\"older's
inequality gives
\[
 \int_{B_\rho} w\psi^2
 \le\|w\|_{L^s(B_\rho)}\|\psi\|_{L^{2s/(s-1)}}^2
 \le\|w\|_{L^s(B_\rho)}\|\psi\|_{L^{2\chi}}^{2\beta}
                     \|\psi\|_{L^2}^{2(1-\beta)}.
\]
For clarity, put
\[
 X=\|\psi\|_{2\chi}^2,\qquad Y=\|\psi\|_2^2,
 \qquad A=Cp\|w\|_{L^s(B_\rho)}.
\]
Weighted Young's inequality gives, for any \(\delta>0\),
\[
 A X^\beta Y^{1-\beta}
 \le \delta X+
 (1-\beta)\beta^{\beta/(1-\beta)}
 \delta^{-\beta/(1-\beta)}A^{1/(1-\beta)}Y.
\]
Indeed, apply \(ab\le\beta a^{1/\beta}+(1-\beta)b^{1/(1-\beta)}\)
to \(a=(\delta/\beta)^\beta X^\beta\) and
\(b=A(\beta/\delta)^\beta Y^{1-\beta}\).
Taking \(\delta=1/2\), and recalling that
\(1/(1-\beta)=2s/(2s-n)\), yields
\begin{equation}\label{eq:3.8}
 Cp\int_{B_\rho} w\psi^2\le\tfrac12\|\psi\|_{2\chi}^2
 +C p^{\frac{2s}{2s-n}}\|w\|_{L^s(B_\rho)}^{\frac{2s}{2s-n}}\|\psi\|_2^2.
\end{equation}
Absorbing \(\frac12\|\psi\|_{2\chi}^2\) into the left-hand side
of~\eqref{eq:3.7} and using \(0\le\eta\le1\), we get
\[
 \|\psi\|_{2\chi}^2
 \le C\left[p^2(\rho-\rho')^{-2}
      +p^{\frac{2s}{2s-n}}\|w\|_{L^s(B_\rho)}^{\frac{2s}{2s-n}}\right]
      \int_{B_\rho}w^p.
\]
Since \(p\ge2\) and \(m=\max\{2,2s/(2s-n)\}\), both \(p^2\) and
\(p^{2s/(2s-n)}\) are bounded by \(p^m\).  Noting that \(\eta\equiv1\) on
\(B_{\rho'}\), we have \(\|w\|_{L^{p\chi}(B_{\rho'})}^p\le\|\psi\|_{2\chi}^2\),
and taking \(p\)-th roots gives~\eqref{eq:3.6}.
\end{proof}

With Proposition~\ref{prop:3.4} in place, we can now prove
Theorem~\ref{thm:main}.

\begin{proof}[Proof of Theorem~\ref{thm:main}]
If \(E_R=0\), then \(u\equiv0\) on \(B_R(p_0)\) and there is nothing to
prove, so we may assume \(0<E_R\le1\).  Fix \(q(3)=q(4)=1/4\) and
\(q(5)=3/5\), and set \(s=2+q(n)>n/2\), \(\chi=n/(n-2)\), and
\(\tau=(1+\theta)/2\).

By rescaling the immersion \(F\) to \(\widetilde F=R^{-1}F\), we may
assume without loss of generality that \(R=1\), so that \(B_R(p_0)=B_1(p_0)\),
\(\int_{B_1}w=E_R\le1\), and \(|H|\le1\) when \(n=5\) (note that strong
stability and \(E_R\) are scale-invariant).  Applying
Corollary~\ref{cor:H-free} for \(n=3,4\) and Theorem~\ref{thm:2.4} for
\(n=5\) on \(B_\tau\subset B_1\), we obtain
\[
\|w\|_{L^s(B_\tau)}\le C(n,\theta)E_R^{1/s}
\le C(n,\theta),
\]
and hence \(K:=\|w\|_{L^s(B_\tau)}^{\frac{2s}{2s-n}}\le C(n,\theta)\).

Now let \(\theta\le r<t\le\tau\), set \(d=t-r<1\), and define for \(j\ge0\)
\[
\rho_j=r+d\,2^{-j},\qquad p_j=2\chi^j,\qquad
N_j=\|w\|_{L^{p_j}(B_{\rho_j})}.
\]
Applying Proposition~\ref{prop:3.4} on the pair of balls
\(B_{\rho_{j+1}}\subset B_{\rho_j}\) with \(\rho_j-\rho_{j+1}=d\,2^{-j-1}\)
gives
\[
N_{j+1}\le
\bigl[C p_j^m(4^{j+1}d^{-2}+K)\bigr]^{1/p_j}N_j
\le\bigl[C_0p_j^m4^{j+1}d^{-2}\bigr]^{1/p_j}N_j,
\]
where \(C_0=C_0(n,\theta)\ge1\).  Iterating this relation from \(j=0\) to
\(k-1\) yields
\[
N_k\le d^{-2\sum_{j=0}^{k-1}1/p_j}
\prod_{j=0}^{k-1}(C_0p_j^m4^{j+1})^{1/p_j}
\|w\|_{L^2(B_t)}.
\]
Since \(\sum_{j=0}^\infty 1/p_j=n/4\) and \(\sum_{j=0}^\infty j/p_j<\infty\),
the infinite product converges: its logarithm is bounded by a constant
times \(\sum_j(1+j)/p_j\). Since \(B_r\subset B_{\rho_k}\) and \(w\)
is smooth on a neighborhood of \(\overline{B_t}\), passing to the limit
in the \(L^{p_k}(B_r)\) norms gives the
\(L^2\)-\(L^\infty\) bound
\[
S(r):=\sup_{B_r}w\le C(t-r)^{-n/2}\|w\|_{L^2(B_t)}
\qquad(\theta\le r<t\le\tau).
\]

To replace the \(L^2\) norm of \(w\) by its \(L^1\) norm \(E_R\), we
interpolate via
\(\int_{B_t}w^2\le S(t)\int_{B_t}w\le S(t)E_R\).
By Cauchy's inequality, for any \(a\in(0,1)\),
\[
S(r)\le C(t-r)^{-n/2}S(t)^{1/2}E_R^{1/2}
\le aS(t)+\frac{C^2}{4a}(t-r)^{-n}E_R.
\]
Choosing \(a=2^{-(n+1)}\) and applying this estimate along the sequence
\(r_j=\theta+(\tau-\theta)(1-2^{-j})\), an induction on \(k\) shows that
\[
S(\theta)\le a^kS(r_k)
+\frac{C^22^{2n-1}E_R}{(\tau-\theta)^n}
\sum_{j=0}^{k-1}(a2^n)^j.
\]
Since \(a2^n=1/2\) and \(S(r_k)\le S(\tau)<\infty\), sending
\(k\to\infty\) yields \(S(\theta)\le C(n,\theta)E_R\).  Scaling back to
radius \(R\) and taking square roots completes the proof of~\eqref{eq:main}.
\end{proof}

\begin{remark}\label{rem:sharp}
The exponent \(1/2\) on \(E_R\) in~\eqref{eq:main} is sharp.  Indeed, near
any nonumbilic point \(p\in M\), a Taylor expansion in normal coordinates
as \(R\downarrow0\) gives
\[
 E_R=\omega_nu(p)^2R^2(1+o(1)),\qquad
 R\sup_{B_{\theta R}(p)}u=u(p)R(1+o(1)),
\]
so for any \(\gamma>1/2\),
\[
 \frac{R\sup_{B_{\theta R}(p)}u}{E_R^\gamma}
 =\omega_n^{-\gamma}u(p)^{1-2\gamma}
       R^{1-2\gamma}(1+o(1))\longrightarrow\infty
 \qquad\text{as }R\downarrow0.
\]
Standard nonumbilic strongly stable examples include open graphical pieces
of round cylinders or of Scherk's minimal surface crossed with \(\R^{n-2}\)
(which are strongly stable since they admit a positive Jacobi
field~\cite{FCS80}).
\end{remark}

\section{Principal-Curvature Degeneration and Finite Morse Index}
\label{sec:4}

In this section, we combine Theorem~\ref{thm:main} with a stability-radius
argument to control regions of small principal curvature when the strong
Morse index is finite. The methodological ingredients are standard:
compare the stability radius in~\cite[Definition~2.4]{AMM24}, the
disjoint-unstable-region index bound in~\cite[Lemma~2.3]{AMM24}, and
the Lipschitz property in~\cite[Lemma~2.5]{AMM24}.
The separate result~\cite[Lemma~4.4]{AMM24} is a covering lemma based
on the folding number and Besicovitch covering, rather than a
stability-radius regularity lemma. Here these ingredients are adapted
to intrinsic balls of a fixed smooth CMC immersion, using a truncated
radius and a compact-set selection.

For related finite-exceptional-set curvature control, see
\cite[Lemma~2.2]{CKM17}, which concerns sequences of embedded minimal
surfaces in three-manifolds and uses point selection, blow-up, and
induction on the index. We do not invoke that curvature estimate in
the proof below. The input specific to the present application is the
combination of the intrinsic energy estimate with the principal-curvature
gap, yielding the stated energy cost and localization scale.

Throughout this section, \(F:M^n\to\R^{n+1}\) (\(3\le n\le5\)) is a
smooth two-sided CMC immersion with \(H>0\) (after choosing the unit
normal appropriately).  We write \(u=|\mathring A|\) and order the
principal curvatures as \(\lambda_1\le\cdots\le\lambda_n\).  Recall that
for an open set \(U\subset M\), \(\operatorname{ind}(U)\) denotes the
maximal dimension of a subspace of \(C_c^\infty(U)\) on which
\[
Q(\varphi)=\int_U\bigl(|\nabla\varphi|^2-|A|^2\varphi^2\bigr)\,dv_M
\]
is negative definite.  All geodesic balls and distances below are taken
with respect to the induced metric on \(M\).

For \(1\le m\le n-1\) and \(0\le\eta<1\), consider the set
\begin{equation}\label{eq:4.1}
\mathcal D_{m,\eta}(K)
=\{x\in K:\lambda_m(x)\le\eta H\}.
\end{equation}
By definition, \(x\in\mathcal D_{m,\eta}(K)\) means that at least \(m\)
principal curvatures at \(x\) are bounded above by \(\eta H\); in
particular, \(\mathcal D_{m,\eta}(K)\) contains the set where at least
\(m\) principal curvatures satisfy \(|\lambda_i|\le\eta H\).  When
\(\eta=0\), \(\mathcal D_{m,0}(K)\) contains all points of \(K\) where
\(A\) has nullity at least \(m\), and \(\mathcal D_{1,0}(K)\) is precisely
the set where \(A\) fails to be positive definite.

\begin{lemma}\label{lem:4.1}
At any point where \(\lambda_m\le\eta H\), we have
\begin{equation}\label{eq:4.2}
u^2\ge\frac{nm}{n-m}(1-\eta)^2H^2,
\end{equation}
with equality if and only if
\[
\lambda_1=\cdots=\lambda_m=\eta H,
\qquad
\lambda_{m+1}=\cdots=\lambda_n=\frac{n-m\eta}{n-m}H.
\]
\end{lemma}

\begin{proof}
Write \(z_i=\lambda_i-H\), so that \(\sum_{i=1}^n z_i=0\) and
\(\sum_{i=1}^m z_i\le-m(1-\eta)H\).  By the Cauchy--Schwarz inequality
applied separately to the first \(m\) and last \(n-m\) terms,
\[
 u^2=\sum_{i=1}^n z_i^2
 \ge\left(\frac1m+\frac1{n-m}\right)\left(\sum_{i=1}^m z_i\right)^2
 \ge\frac{nm}{n-m}(1-\eta)^2H^2,
\]
and the condition for equality is immediate.
\end{proof}

\begin{lemma}\label{lem:4.2}
Suppose \(B_R(p)\Subset M\) and
\(\operatorname{ind}(B_R(p))\le I\) for some \(I\in\mathbb N_0\).
Let \(K=\overline{B_{R/2}(p)}\).
Then there exists a finite set \(\mathcal S\subset K\) with
\(\#\mathcal S\le I\) such that for every \(x\in K\setminus\mathcal S\),
the ball \(B_{2\rho(x)}(x)\) is strongly stable, where
\begin{equation}\label{eq:4.3}
\rho(x)=\frac18\min\{R,d_M(x,\mathcal S)\}
\end{equation}
(with the convention \(d_M(x,\varnothing)=+\infty\)).
\end{lemma}

\begin{proof}
For each \(x\in K\), define the truncated stability radius
\[
 s(x)=\sup\{0<r\le R/4:B_r(x)\text{ is strongly stable}\}.
\]
Fix \(x\in K\). In a sufficiently small normal coordinate ball,
the metric and volume element are uniformly comparable to their
Euclidean counterparts. Thus, for some \(r_x>0\) and \(C_x<\infty\),
\[
 \int_{B_r(x)}\varphi^2\le C_xr^2
 \int_{B_r(x)}|\nabla\varphi|^2,
 \qquad 0<r\le r_x,\quad\varphi\in C_c^\infty(B_r(x)).
\]
Let \(A_x^2=\sup_{B_{r_x}(x)}|A|^2<\infty\), shrinking the
coordinate ball to have compact closure if necessary. Then
\[
 \int_{B_r(x)}|A|^2\varphi^2
 \le A_x^2C_xr^2\int_{B_r(x)}|\nabla\varphi|^2.
\]
Choosing \(r\) so that \(A_x^2C_xr^2\le1\) proves \(s(x)>0\);
if \(A_x=0\), this restriction is unnecessary. Moreover, if
\(d=d_M(x,y)<s(x)\), then \(B_{r-d}(y)\subset B_r(x)\) for all
\(r\in(d,s(x))\), which implies \(s(y)\ge s(x)-d_M(x,y)\).  By symmetry,
\(s\) is \(1\)-Lipschitz on \(K\):
\[
 |s(x)-s(y)|\le d_M(x,y).
\]
Note also that the ball \(B_{s(x)}(x)\) itself is strongly stable, since
any test function \(\phi\in C_c^\infty(B_{s(x)}(x))\) is supported in
\(B_r(x)\) for some \(r<s(x)\).

Since \(B_R(p)\Subset M\), the set
\(K=\overline{B_{R/2}(p)}\) is compact and lies in \(B_R(p)\).
The positive continuous function \(s\) therefore has a positive minimum
on \(K\).

We now select \(\mathcal S\) inductively. Start with \(K_0=K\). At step
\(k\ge1\), we stop if \(K_{k-1}=\varnothing\), or if
\(K_{k-1}\ne\varnothing\) and \(\min_{K_{k-1}}s=R/4\).  Otherwise, we pick a
point \(a_k\in K_{k-1}\) where \(s|_{K_{k-1}}\) attains its minimum
\(s(a_k)<R/4\), and set
\[
 K_k=\bigl\{x\in K_{k-1}:d_M(x,a_k)\ge2s(x)+2s(a_k)\bigr\}.
\]
Since \(s\) is continuous, each \(K_k\) is compact.  By construction, for
any \(i<j\), \(a_j\in K_i\) and hence
\[
 d_M(a_i,a_j)\ge2s(a_i)+2s(a_j).
\]
Consequently, the balls \(B_{2s(a_i)}(a_i)\) are pairwise disjoint.
Since \(a_i\in K\) and \(2s(a_i)<R/2\), each \(B_{2s(a_i)}(a_i)\) is
contained in \(B_R(p)\) and is unstable (otherwise \(s(a_i)\) would be at
least \(\min\{2s(a_i),R/4\}>s(a_i)\)).  Each \(B_{2s(a_i)}(a_i)\)
therefore supports a test function \(\phi_i\) with \(Q(\phi_i)<0\).
Extend each \(\phi_i\) by zero to \(B_R(p)\). Disjoint supports imply
\[
 Q\left(\sum_i b_i\phi_i\right)=\sum_i b_i^2Q(\phi_i)<0
 \quad\text{whenever }(b_i)_i\ne0.
\]
The functions are linearly independent, so \(I+1\) selections would
contradict \(\operatorname{ind}(B_R(p))\le I\). The process therefore
terminates after \(N\le I\) selections, including \(N=0\) if it
stops immediately.  We set \(\mathcal S=\{a_1,\ldots,a_N\}\).

Finally, take any \(x\in K\setminus\mathcal S\).  If \(x\) was removed at
step \(k\), then \(x\in K_{k-1}\), so \(s(a_k)\le s(x)\) by the choice of
\(a_k\), and thus
\[
 d_M(x,\mathcal S)\le d_M(x,a_k)<2s(x)+2s(a_k)\le4s(x).
\]
If instead \(x\) remained when the algorithm terminated, then
\(s(x)=R/4\).  In either case,
\[
 s(x)\ge\frac14\min\{R,d_M(x,\mathcal S)\}=2\rho(x),
\]
so \(B_{2\rho(x)}(x)\subset B_{s(x)}(x)\) is strongly stable.
The selection used only \(s\) and the geometry of \(B_R(p)\); it
does not depend on \(m,\eta\), or a later choice of radius. Every
smaller ball \(B_{2r}(x)\), \(0<r\le\rho(x)\), is stable as well.
\end{proof}

We now complete the proof of Theorem~\ref{thm:1.2}.  Below we write
\[
 C_*=C(n,1/2),\qquad
 c_{m,\eta}=\frac{nm}{n-m}(1-\eta)^2.
\]

\begin{proof}[Proof of Theorem~\ref{thm:1.2}]
Let \(\mathcal S\subset K\) be the set given by Lemma~\ref{lem:4.2}.  Take
\(x\in\mathcal D_{m,\eta}(K)\setminus\mathcal S\) and
\(0<r\le\rho(x)\).  By Lemma~\ref{lem:4.2}, \(B_{2r}(x)\) is strongly
stable, and \(B_{2r}(x)\subset B_{R/4}(x)\subset B_R(p)\Subset M\).

If \(r^{2-n}\int_{B_r(x)}u^2\le1\), then since \(Hr\le HR\le1\), we may
apply Theorem~\ref{thm:main} (with \(\theta=1/2\)) on \(B_{r}(x)\) to get
\[
u(x)^2\le C_*^2r^{-n}\int_{B_r(x)}u^2\,dv_M.
\]
Consequently, Lemma~\ref{lem:4.1} gives
\[
 \int_{B_r(x)}u^2\,dv_M
 \ge C_*^{-2}u(x)^2r^n
 \ge C_*^{-2}c_{m,\eta}H^2r^n
 \ge\kappa_nc_{m,\eta}H^2r^n,
\]
which is~\eqref{eq:1.3}. The application of Theorem~\ref{thm:main}
is local: all its tests are supported in the stable region
\(B_{2r}(x)\), and \(B_r(x)\Subset B_{2r}(x)\).
On the other hand, if \(r^{2-n}\int_{B_r(x)}u^2>1\), then using \(Hr\le1\)
and \(c_{m,\eta}\le n(n-1)\le\kappa_n^{-1}\), we directly have
\[
\int_{B_r(x)}u^2>r^{n-2}\ge H^2r^n
\ge\kappa_n c_{m,\eta}H^2r^n,
\]
so~\eqref{eq:1.3} holds in both cases. In the second branch,
\(Hr\le1\) is used even when \(n=3,4\); this explains the scale
assumption in Theorem~\ref{thm:1.2}. Notice also that this argument
proves the same energy lower bound on any stable ball \(B_{2r}(x)\)
contained in \(B_R(p)\), with \(x\in\mathcal D_{m,\eta}(K)\) and
\(Hr\le1\), regardless of whether its center or radius arises from
Lemma~\ref{lem:4.2}. We will use this observation in the corollary.

To deduce~\eqref{eq:1.5}, note that if \(E_R=0\), then \(u\equiv0\) on
\(B_R(p)\) and \(\mathcal D_{m,\eta}(K)=\varnothing\).  Otherwise,
evaluating~\eqref{eq:1.3} at \(r=\rho(x)\) for any
\(x\in\mathcal D_{m,\eta}(K)\setminus\mathcal S\) and bounding the
left-hand side by \(\int_{B_R(p)}u^2=R^{n-2}E_R\) yields
\[
 \rho(x)^n\le\frac{R^{n-2}E_R}{\kappa_nc_{m,\eta}H^2}
 =R^n\frac{E_R}{\kappa_nc_{m,\eta}(HR)^2}
 =\ell_{m,\eta}^n.
\]
When \(\ell_{m,\eta}<R/8\), the definition
of \(\rho(x)\) forces \(d_M(x,\mathcal S)=8\rho(x)\le8\ell_{m,\eta}\),
which proves~\eqref{eq:1.5} for points outside \(\mathcal S\).
Points in \(\mathcal D_{m,\eta}(K)\cap\mathcal S\) belong to the
union automatically. If \(\mathcal S=\varnothing\), then
\(\rho(x)=R/8\); the strict inequality \(\ell_{m,\eta}<R/8\)
therefore excludes every point of \(\mathcal D_{m,\eta}(K)\).
For \(m=1,\eta=0\), being outside this set means \(\lambda_1>0\),
which proves the assertion about positive definiteness.
\end{proof}

When \(\mathcal D_{m,\eta}(K)\) contains connected components of positive
diameter, packing disjoint balls along each component yields an energy
lower bound in terms of their diameters.

\begin{corollary}\label{cor:4.3}
Under the assumptions of Theorem~\ref{thm:1.2}, let
\(\Gamma\subset\mathcal D_{m,\eta}(K)\) be a nonempty connected set of
intrinsic diameter \(L=\operatorname{diam}_M\Gamma\).  Then
\begin{equation}\label{eq:4.5}
E_R\ge\frac{\kappa_n}{8^n}\frac{nm}{n-m}(1-\eta)^2(HR)^2
\frac{(L/R)^n}{(I+1)^{n-1}}.
\end{equation}
More generally, let \(\Gamma_1,\ldots,\Gamma_N\subset K\) be connected
sets of positive diameters \(L_1,\ldots,L_N\) with
\(\Gamma_j\subset\mathcal D_{m_j,\eta_j}(K)\), such that the tubular
neighborhoods
\[
 U_j=\{x\in M:d_M(x,\Gamma_j)<L_j/4\},\qquad 1\le j\le N,
\]
are pairwise disjoint.  Setting \(I_j=\operatorname{ind}(U_j)\), we have
\(\sum_{j=1}^N I_j\le I\) and
\begin{equation}\label{eq:4.joint-local}
 E_R\ge\frac{\kappa_n}{8^n}(HR)^2
 \sum_{j=1}^N
 \frac{c_{m_j,\eta_j}(L_j/R)^n}{(I_j+1)^{n-1}}.
\end{equation}
In particular,
\begin{equation}\label{eq:4.joint}
 E_R\ge\frac{\kappa_n}{8^n}(HR)^2
 \frac{\left(\sum_{j=1}^N c_{m_j,\eta_j}^{1/n}L_j/R\right)^n}
 {(I+N)^{n-1}},
\end{equation}
where \(c_{m,\eta}=nm(1-\eta)^2/(n-m)\).
\end{corollary}

\begin{proof}
If \(L=0\), the single-set assertion~\eqref{eq:4.5} is immediate
from \(E_R\ge0\). Hence assume \(L>0\). It suffices to prove the
family statements~\eqref{eq:4.joint-local} and~\eqref{eq:4.joint}:
for \(N=1\), the local index \(I_1\le I\), and replacing \(I_1\)
by \(I\) only decreases the lower bound. The sets need only be
connected; neither path connectedness nor attainment of their
diameters is assumed.
Since \(\Gamma_j\subset K=\overline{B_{R/2}(p)}\), we have \(L_j\le R\)
and \(U_j\subset B_{3R/4}(p)\Subset B_R(p)\).  Because the open sets
\(U_j\) are pairwise disjoint, taking the direct sum of maximal
\(Q\)-negative subspaces of \(C_c^\infty(U_j)\) immediately gives
\(\sum_{j=1}^N I_j\le\operatorname{ind}(B_R(p))\le I\).

Fix \(j\) and any \(0<L_j'<L_j\), and choose \(z,z'\in\Gamma_j\) with
\(d_M(z,z')>L_j'\).  Since \(\Gamma_j\) is connected, the image of
\(\Gamma_j\) under \(x\mapsto d_M(z,x)\) contains \([0,L_j']\), so we can
find \(2I_j+3\) points \(x_0,\ldots,x_{2(I_j+1)}\in\Gamma_j\) such that
\[
 d_M(z,x_k)=\frac{kL_j'}{2(I_j+1)},\qquad
 r_j=\frac{L_j'}{8(I_j+1)}.
\]
By the triangle inequality, \(d_M(x_k,x_\ell)\ge4r_j\) for \(k\ne\ell\),
so the \(2I_j+3\) balls \(B_{2r_j}(x_k)\) are pairwise disjoint and
contained in \(U_j\) (as \(2r_j<L_j/4\)).  Since
\(\operatorname{ind}(U_j)=I_j\), at most \(I_j\) of the balls
\(B_{2r_j}(x_k)\) can be unstable, leaving at least \(I_j+3\ge I_j+1\)
strongly stable balls \(B_{2r_j}(x_k)\).

On each such stable ball, the dichotomy in the proof of
Theorem~\ref{thm:1.2} applies to \(B_{r_j}(x_k)\) and gives
\[
 \int_{B_{r_j}(x_k)}u^2
 \ge\kappa_n c_{m_j,\eta_j}H^2r_j^n.
\]
Summing over \(I_j+1\) of these disjoint balls and letting
\(L_j'\uparrow L_j\) yields
\[
 \int_{U_j}u^2\ge\frac{\kappa_n}{8^n}
      c_{m_j,\eta_j}H^2\frac{L_j^n}{(I_j+1)^{n-1}}.
\]
Summing over \(j=1,\ldots,N\) and using the disjointness of the \(U_j\)
inside \(B_R(p)\) gives~\eqref{eq:4.joint-local}.  Finally, by H\"older's
inequality with exponents \(n\) and \(n/(n-1)\),
\[
 \sum_{j=1}^N c_{m_j,\eta_j}^{1/n}\frac{L_j}{R}
 \le\left(\sum_{j=1}^N
      \frac{c_{m_j,\eta_j}(L_j/R)^n}{(I_j+1)^{n-1}}\right)^{1/n}
      \left(\sum_{j=1}^N(I_j+1)\right)^{(n-1)/n},
\]
and since \(\sum_{j=1}^N(I_j+1)\le I+N\), raising both sides to the
\(n\)-th power and combining with~\eqref{eq:4.joint-local}
gives~\eqref{eq:4.joint}.
\end{proof}

\begin{remark}
The generalized cylinder \(S^{n-m}(a)\times\R^m\) achieves equality in
Lemma~\ref{lem:4.1} for \(\eta=0\), since
\[
 H=\frac{n-m}{na}\qquad\text{and}\qquad
 |\mathring A|^2=\frac{m(n-m)}{na^2}=\frac{nm}{n-m}H^2.
\]
Note that for \(1\le m\le n-1\), the complete cylinder \(S^{n-m}(a)\times\R^m\)
has infinite Morse index (as one sees by translating compactly supported
unstable test functions along the \(\R^m\) factor), whereas
Theorem~\ref{thm:1.2} and Corollary~\ref{cor:4.3} apply to relatively
compact regions of finite index. To see the index assertion directly,
take \(\varphi(y,z)=\chi(z/T)\), with a fixed nonzero
\(\chi\in C_c^\infty(\R^m)\). Since \(|A|^2=(n-m)/a^2\),
\[
 Q(\varphi)=|S^{n-m}(a)|
 \left(T^{m-2}\int_{\R^m}|\nabla\chi|^2
       -\frac{n-m}{a^2}T^m\int_{\R^m}\chi^2\right)<0
\]
for sufficiently large \(T\). Arbitrarily many disjoint translates
give arbitrarily large negative subspaces. Thus the cylinder is an
equality model for the algebraic gap in Lemma~\ref{lem:4.1}; no claim
of equality in the localization or energy--diameter estimates follows.
\end{remark}

\bibliographystyle{alpha}
\bibliography{holder_closure}
\enlargethispage{2\baselineskip}

\end{document}